\documentclass[12pt,a4paper,reqno]{amsart}
\usepackage{amsfonts}
\usepackage{amsthm}
\usepackage{amsmath}
\usepackage{amscd}
\usepackage[latin2]{inputenc}
\usepackage{t1enc}
\usepackage[mathscr]{eucal}
\usepackage{indentfirst}
\usepackage{graphicx}
\usepackage{graphics}
\usepackage{pict2e}
\usepackage{epic}
\numberwithin{equation}{section}
\usepackage[margin=2.8cm]{geometry}
\usepackage{hyperref}
\hypersetup{
    colorlinks=true,
    linkcolor=blue,
    filecolor=blue,      
    urlcolor=cyan,
}

\usepackage{epstopdf} 
\usepackage{tabularx}
\usepackage{longtable}
\usepackage{multirow}
\usepackage{hhline}
\usepackage{fancyhdr}
\usepackage{stmaryrd}
\usepackage[all]{xy}
\usepackage{xcolor}
\usepackage{comment}
\usepackage{shuffle}
\usepackage{nameref, cleveref}
\usepackage{mathtools}
\usepackage{enumerate}
\usepackage{tikz-cd}

\theoremstyle{plain}
\newtheorem{Th}{Theorem}[section]

\newtheorem{Cor}[Th]{Corollary}

\theoremstyle{definition}
\newtheorem{Def}[Th]{Definition}
\newtheorem{Conj}[Th]{Conjecture}
\newtheorem{Rem}[Th]{Remark}
\newtheorem{?}[Th]{Problem}

\begin{document}

\title[Universal $\mathbb{F}_q-$Family of $v-$adic MZVs]{Universal $\mathbb{F}_q-$Family of $v-$adic multiple zeta Values over Function Fields}

\author[Qibin Shen]{Qibin Shen}

\address{Department of Mathematics\\University of Rochester \\ 500 Joseph C. Wilson Blvd., Rochester, NY, 14627} 

\email{qshen4@ur.rochester.edu}

\date{September 20, 2019}



\begin{abstract} This paper aims to study the $\mathbb{F}_q-$linear relations between interpolated $v-$adic multiple zeta values over function fields. We proved a universal family of linear relations of interpolated $v-$adic MZVs, which is conjectured to generate all linear relations over $\mathbb{F}_q$.
At the end of the paper, we also show that these relations can be generalized to all multiple harmonic type sums.
\end{abstract}

\maketitle

\section{Introduction} 

Multiple Zeta Values (MZVs) was originally introduced and studied by Euler, and recently, these values showed up
again in various subjects in mathematics and
mathematical physics. In Furusho's paper \cite{F04}, by making an analytic continuation of $p-$adic multiple polylogarithms introduced by Coleman's $p-$adic iterated integration theory \cite{Col} and he was able to define $p-$adic multiple
zeta values to be a limit value at $1$ of analytically continued $p-$adic multiple
polylogarithms.

It's well known that researches on function fields and number fields often go in parallel and thus mutually beneficial. Often the former one is inspired by the latter.

Unlike the classical case, the interpolated $v-$adic Multiple Zeta Values ($v-$adic MZVs) in function field case are well-defined at all integer points. In this paper, we mainly studied $v-$adic multiple zeta values over function fields. At the end of this article we will give a more general result and give some direct application to other types of multiple zeta values.

Now we first adapt some notations before we give the definition.

\textbf{Notations:}
\begin{eqnarray*}
\mathbb{Z}_- &=& \{ \textit{negative integers} \},\\
q &=& p^f, \textit{a power of a prime p},\\
\mathbb{F}_q &=& \textit{a finite field of $q$ elements},\\
K &=& \textit{function field over $\mathbb{F}_q$},\\
\infty &=& \textit{a rational place in $K$},\\
A &=& \textit{the ring of integral elements outside $\infty$},\\
v &=& \textit{a monic prime in $A$},\\
A_{d+} &=& \textit{monics in $A$ of degree $d$},\\
K_\infty &=& \textit{completion of $K$ at $\infty$},\\
K_v &=& \textit{completion of $K$ at prime $v$}.
\end{eqnarray*}

\begin{Def}
An integer $s$ is called \textit{$q-$even} if $(q-1) | s$. Otherwise, it's called \textit{$q-$odd}.
\end{Def}
The reason to introduce ``$q-$even'' is that the behavior of Carlitz zeta values ($\zeta(s)$ of depth 1 defined in \Cref{mzv}) at $q-$even integers is similar as that of Riemann zeta values at even integers.

Now we are ready to define MZVs over $K$. Given $k \in \mathbb{Z}$ and $\mathbf{s}$ with $s_j \in \mathbb{Z}$, for $d \geq 0$, let
$$S_d(k):=\sum_{a \in A_{d+}} \frac{1}{a^k} \in K,$$
$$S_d(\mathbf{s}):=S_d(s_1)\sum_{d>d_2>\cdots>d_r \geq 0} S_{d_2}(s_2) \cdots S_{d_r}(s_r) \in K,$$
$$S^*_d(\mathbf{s}):=S_d(s_1)\sum_{d\geq d_2\geq \cdots\geq d_r \geq 0} S_{d_2}(s_2) \cdots S_{d_r}(s_r) \in K.$$
\begin{Def}\label{mzv}
Let $\mathbf{s}=(s_1, \ldots, s_r), s_i \in \mathbb{Z}$, define \textit{multiple zeta value} 
$$\zeta(\mathbf{s}):=\sum_{d_1 > \cdots > d_r \geq 0} S_{d_1}(s_1) \cdots S_{d_r}(s_r) \in K_\infty.$$
We call $r$ the depth of $\zeta(\mathbf{s})$, and if each $s_i>0$, $wt:=\sum_{i=1}^rs_i$ is the weight of $\zeta(\mathbf{s})$.
\end{Def}

Given $K$ a function field, $k \in \mathbb{Z}$, $v$ a finite prime in $A_+$ and $\mathbf{s}=(s_1,\cdots,s_r)$ with each $s_i \in \mathbb{Z}$, for $d \geq 0$. We let
$$\widetilde{S}_d(k):=\sum_{\substack{a \in A_{d+}\\ (a, v)=1}} \frac{1}{a^k} \in K,$$
and
$$\widetilde{S}_d(\mathbf{s}):=\widetilde{S}_d(s_1)\sum_{d>d_2>\cdots>d_r \geq 0} \widetilde{S}_{d_2}(s_2) \cdots \widetilde{S}_{d_r}(s_r) \in K,$$
$$\widetilde{S}^{\star}_d(\mathbf{s}):=\widetilde{S}_d(s_1)\sum_{d\geq d_2\geq \cdots\geq d_r \geq 0} \widetilde{S}_{d_2}(s_2) \cdots \widetilde{S}_{d_r}(s_r) \in K.$$

We are now able to introduce the interpolated $v-$adic multiple zeta values and $v-$adic multiple zeta star values defined first by Thakur in \cite{T04} here.
\begin{Def}
Let $\mathbf{s}=(s_1, \ldots, s_r), s_i \in \mathbb{Z}$, $v$ a prime in $A_+$, define \textit{$v-$adic multiple zeta values} 
$$\zeta_v(\mathbf{s}):=\sum_{d_1 > \cdots > d_r \geq 0} \widetilde{S}_{d_1}(s_1) \cdots \widetilde{S}_{d_r}(s_r) \in K_v.$$
And we may define \textit{$v-$adic multiple zeta star values} as
$$\zeta^{\star}_v(\mathbf{s}):=\sum_{d_1 \geq \cdots \geq d_r \geq 0} \widetilde{S}_{d_1}(s_1) \cdots \widetilde{S}_{d_r}(s_r) \in K_v.$$
\end{Def}
It's well known that these power series is convergent for any $\mathbf{s}$ with each $s_i\in \mathbb{Z}$ and any prime $v$. These values can be extended continuously to a much bigger domain $S_v$, but we will restrict ourselves in the special values at $\mathbf{s}$ with integer coordinates. The more general case of $S_v$ works similarly.

Similar to the finite MZVs in classical case, here we also define the function field version of finite MZVs and finite MZSVs.

\begin{Def}
Let $K$ be a given function field, we define $$\mathcal{A}=\frac{\prod_v\, \mathbb{F}_v}{\oplus_v\, \mathbb{F}_v},$$
where $v$ runs over all monic irreducible polynomials in $\mathbb{F}_q[t]$ and $\mathbb{F}_v:=A/(v)$. $\mathcal{A}$ is an $\mathbb{F}_q(t)-$algebra. 

Then, we define the finite multiple zeta values and finite multiple zeta star values as
$\zeta^{\mathcal{A}}(\mathbf{s}):=(\zeta_v^{\mathcal{A}} (\mathbf{s}))_{\textit{$v$ prime}} \in \mathcal{A}$. $\zeta^{\mathcal{A}\star}(\mathbf{s}):=(\zeta_v^{\mathcal{A}\star} (\mathbf{s}))_{\textit{$v$ prime}} \in \mathcal{A}$,
where
$$\zeta_{v}^{\mathcal{A}}(\mathbf{s}):=\sum_{deg(v)>d_1 > \cdots > d_r \geq 0} S_{d_1}(s_1) \cdots S_{d_r}(s_r) \mod v,$$
$$\zeta_{v}^{\mathcal{A}\star}(\mathbf{s}):=\sum_{deg(v)>d_1 \geq \cdots \geq d_r \geq 0} S_{d_1}(s_1) \cdots S_{d_r}(s_r) \mod v.$$
\end{Def}

\begin{Def}
Given $K=\mathbb{F}_q(t)$, $v$ a monic prime with degree $d$.
We define
$S_v:=\lim\limits_{\substack{\longleftarrow\\ n}}\mathbb{Z}/(q^d-1)p^n\mathbb{Z}$, and we say $s=c_{-1}+\sum_{i=0}^{\infty}c_i(q^d-1)p^i\in S_v$
 where $q^d-1<c_{-1}\leq 0, p<c_i\leq 0$ is $q-$even iff $q-1|c_{-1}$.
\end{Def}

\begin{Def}\label{trivialzero}\cite{Shen19}
Let $K=\mathbb{F}_q(t)$, each $s_j \in S_v$, $r>1$. $\zeta_v(s_1, \ldots, s_r)=0$ trivially iff either one of the following conditions holds
\begin{itemize}
    \item[(1)] $\exists1 \leq i \leq r$ such that $s_i<0$, $r-i > L_{-s_i} + deg(v)$.
    \item[(2)] $\exists1 \leq i,j\leq r$ such that $s_i,s_j<0$, $deg(v)>r-i > L_{-s_i}$, and $i-j>L_{-s_j}$.
\end{itemize}
We call such zeros \textit{trivial zeros}. Other zeros are called \textit{nontrivial}.

In particular, when $deg(v)=1$, we have $\mathbf{s}$ is a trivial zero iff there exists $1\leq i\leq r$ such that $r-i>L_{-s_i}+1$.
\end{Def}

\begin{Th}\cite{Shen19}
Given a rational function field $K=\mathbb{F}_q(t)$, a degree $1$ prime $v$, $\mathbf{s}=(s_1, \ldots, s_r) \in S_v^r$, $\zeta_v(\mathbf{s})=0$ if and only if one of the following conditions holds
\begin{itemize}
    \item[(1)] $r>1$, $\mathbf{s}$ is a trivial zero,
    \item[(2)] $r=1$, $s$ is $q-$even. 
\end{itemize}
\end{Th}

\begin{Conj}\cite{Shen19}
Given $K=\mathbb{F}_q(t)$, $v$ any monic prime, $\mathbf{s}=(s_1, \ldots, s_r) \in S_v^r$, $\zeta_v(\mathbf{s})=0$ if and only if one of the following conditions holds
\begin{itemize}
    \item[(1)] $r>1$, $\mathbf{s}$ is a trivial zero,
    \item[(2)] $r=1$, $s$ is $q-$even. 
\end{itemize}
\end{Conj}

\begin{Rem}
To keep analogy with Riemann zeta values and to keep the usual terminology, we also call depth one zeros at even integers as trivial, by a slight abuse of the terminology. 
\end{Rem}

\section{Universal $\mathbb{F}_q-$Family of MZVs}
Now we study the $\mathbb{F}_q-$linear relations of $v-$adic MZVs. In this section, we will mainly study a family of linear relations of $v-$adic MZVs with coefficients in $\mathbb{F}_q$ which hold for all primes $v$ and all function fields with rational infinity place. We call such kind of family universal $\mathbb{F}_q-$family. In fact, we noticed that these relations can actually be applied to more general cases.

{\bf Here and below, rather than repeating complicated sum expressions, 
if, say,  the second sum of the previous expression  is unchanged in a given equality, we just denote it by (temporary notation) $\sum_2$}.
\begin{Th}\label{Thm2}
$K$ is a function field over $\mathbb{F}_q$ with a rational infinity place, $v$ is a prime, $\mathbf{s}=(s_i)_{i=1}^n$. 
If $s_i\in \mathbb{Z}$ are distinct and $q-$even and $n$ is odd, then we have:
\begin{eqnarray*}
\sum_{\sigma\in S_n}sgn(\sigma)\zeta_v(\sigma(\mathbf{s}))=0
\end{eqnarray*}
where $S_n$ is the symmetry group and sgn is the sign function on $S_n$.
\end{Th}
\begin{proof}
Using the fact that all $s_i's$ are $q-$even, we have $\zeta_v(s_i)=\sum_{d_{s_i}=0}^{\infty}\widetilde{S}_{d_{s_i}}(s_i)=0.$
Hence, this theorem holds for $n=1$ and we get
\begin{eqnarray}
\widetilde{S}_{0}(s_i)=-\sum_{d_{s_i}=1}^{\infty}\widetilde{S}_{d_{s_i}}(s_i).\label{E1}
\end{eqnarray}

We have
\begin{eqnarray*}
&&\sum_{\sigma\in S_n}sgn(\sigma)\zeta_v(\sigma(\mathbf{s}))\\
&=&\sum_{\substack{\sigma\in S_n\\ d_{s_{\sigma(1)}}>\cdots>d_{s_{\sigma(n)}}\geq 0}}sgn(\sigma)\prod_{i=1}^n\widetilde{S}_{d_{s_{\sigma(i)}}}(s_{\sigma(i)})\\
&=&\sum_{\substack{\sigma\in S_n\\ d_{s_{\sigma(1)}}>\cdots>d_{s_{\sigma(n)}}> 0}}sgn(\sigma)\prod_{i=1}^n\widetilde{S}_{d_{s_{\sigma(i)}}}(s_{\sigma(i)})+\sum_{\substack{\sigma\in S_n\\ d_{s_{\sigma(1)}}>\cdots>d_{s_{\sigma(n)}}= 0}}sgn(\sigma)\prod_{i=1}^n\widetilde{S}_{d_{s_{\sigma(i)}}}(s_{\sigma(i)})\\
&=&{\sum}_1+\sum_{\substack{\sigma\in S_n\\ d_{s_{\sigma(1)}}>\cdots>d_{s_{\sigma(n-1)}}>0}}sgn(\sigma)\prod_{i=1}^{n-1}\widetilde{S}_{d_{s_{\sigma(i)}}}(s_{\sigma(i)})\left(-\sum_{d_{s_{\sigma(n)}}=1}^{\infty}\widetilde{S}_{d_{s_{\sigma(n)}}}(s_{\sigma(n)})\right)\\
&=&{\sum}_1-\sum_{i=1}^{n}\sum_{\substack{\sigma\in S_n\\ d_{s_{\sigma(1)}}>\cdots>d_{s_{\sigma(n)}}>d_{s_{\sigma(i)}}>\cdots>d_{s_{\sigma(n-1)}}> 0}}sgn(\sigma)\prod_{i=1}^n\widetilde{S}_{d_{s_{\sigma(i)}}}(s_{\sigma(i)})\\
&+&\sum_{i=1}^{n-1}-\sum_{\substack{\sigma\in S_n\\ d_{s_{\sigma(1)}}>\cdots>d_{s_{\sigma(i)}}= d_{s_{\sigma(n)}}>\cdots>0}}sgn(\sigma)\prod_{i=1}^{n}\widetilde{S}_{d_{s_{\sigma(i)}}}(s_{\sigma(i)})\\
&=&\sum_{i=1}^{n-1}-\sum_{\substack{\sigma\in S_n\\ d_{s_{\sigma(1)}}>\cdots>d_{s_{\sigma(n)}}>d_{s_{\sigma(i)}}>\cdots>d_{s_{\sigma(n-1)}}> 0}}sgn(\sigma)\prod_{i=1}^n\widetilde{S}_{d_{s_{\sigma(i)}}}(s_{\sigma(i)})\\
&-&\sum_{i=1}^{n-1}\sum_{\substack{\sigma\in S_n\\ d_{s_{\sigma(1)}}>\cdots>d_{s_{\sigma(i)}}= d_{s_{\sigma(n)}}>\cdots>0}}sgn(\sigma)\prod_{i=1}^{n}\widetilde{S}_{d_{s_{\sigma(i)}}}(s_{\sigma(i)}).
\end{eqnarray*}

The third equality follows by (\ref{E1}). The fourth equality holds since we can break the second term in the third line into two parts listed in the fourth and fifth lines. The fifth equality is just adding ${\sum}_1$ to latter term in the fourth line.

Hence, we have
\begin{eqnarray*}
&&\sum_{\sigma\in S_n}sgn(\sigma)\zeta_v(\sigma(\mathbf{s}))\\
&=&\sum_{i=1}^{n-1}(-1)^{n-i}\sum_{\substack{\sigma\in S_n\\ d_{s_{\sigma(1)}}>\cdots>d_{s_{\sigma(n)}}> 0}}sgn(\sigma)\prod_{i=1}^n\widetilde{S}_{d_{s_{\sigma(i)}}}(s_{\sigma(i)})\\
&-&\sum_{i=1}^{n-1}\sum_{\substack{\sigma\in S_n\\ d_{s_{\sigma(1)}}>\cdots>d_{s_{\sigma(i)}}= d_{s_{\sigma(n)}}>\cdots>0}}sgn(\sigma)\prod_{i=1}^{n}\widetilde{S}_{d_{s_{\sigma(i)}}}(s_{\sigma(i)})\\
&=&\sum_{\substack{1\leq i\leq n-1\\ i \ odd}}(1-1)\sum_{\substack{\sigma\in S_n\\ d_{s_{\sigma(1)}}>\cdots>d_{s_{\sigma(n)}}> 0}}sgn(\sigma)\prod_{i=1}^n\widetilde{S}_{d_{s_{\sigma(i)}}}(s_{\sigma(i)})-{\sum}_2\\
&=&-{\sum}_2\\
&=&-\sum_{i=1}^{n-1}\sum_{\substack{\sigma\in S_{n}\\ d_{s_{\sigma(1)}}>\cdots>d_{s_{\sigma(i)}}= d_{s_{\sigma(n)}}>\cdots> d_{s_{\sigma(n-1)}}>0\\ \sigma(i)<\sigma(n)}}sgn(\sigma)\prod_{i=1}^{n}\widetilde{S}_{d_{s_{\sigma(i)}}}(s_{\sigma(i)})\\
&-&\sum_{i=1}^{n-1}\sum_{\substack{\sigma\in S_{n}\\ d_{s_{\sigma(1)}}>\cdots>d_{s_{\sigma(n)}}= d_{s_{\sigma(i)}}>\cdots >d_{s_{\sigma(n-1)}}>0\\ \sigma(i)<\sigma(n)}}sgn((i,n))sgn(\sigma)\prod_{i=1}^{n}\widetilde{S}_{d_{s_{\sigma(i)}}}(s_{\sigma(i)})\\
&=&0.
\end{eqnarray*}

Each time we exchange $\sigma(n)$ and $\sigma(i),\cdots,\sigma(n-1)$ for every $\sigma\in S_n$, $sgn(\sigma)$ will be multiplied by $-1$, which implies the first equality. 
The second equality follows by perfectly grouping the consecutive terms $i,i+1$, (it can be done since $n$ is odd hence $n-1$ is even). 
The fourth equality holds since the ${\sum}_2$ breaks up to $\sigma(i)<\sigma(n)$ and $\sigma(i)>\sigma(n)$, which can be expressed as in there because we can always exchange $i,n$ by one transposition $(i,n)$.
\end{proof}

\begin{Th}\label{Thm3}
$K$ is a function field over $\mathbb{F}_q$ with a rational infinity place. If $\mathrm{char}(\mathbb{F}_q)=2$, $v$ a prime, $n\geq 1$, $s_i$ and $2s_i$ are all distinct $($\textit{where $2s_i$ is included only when $k_i>1$}$)$, $q-$even and $k_i\geq 1$, $\forall i\leq n$. Let $\phi:=\sum_{i=1}^nk_i$, and
$\mathbf{s}_0:=((s_1)_{k_1},\cdots,(s_n)_{k_n}).$

For $i$ such that $k_i> 1$, we define
$\mathbf{s}_i:=((s_1)_{k_1},\cdots,(s_{i})_{k_i-2},2s_i,\cdots,(s_n)_{k_n}).$
Then the following holds:
\begin{eqnarray*}
\sum_{{\textit{$i$ where $k_i>1$}}}\left(\sum_{\sigma\in \textbf{L}_i}\zeta_v(\sigma(\mathbf{s}_i))\right)+\phi\sum_{\sigma\in \textbf{L}_0}\zeta_v(\sigma(\mathbf{s}_0))=0,
\end{eqnarray*}
where $\textbf{L}_i$ is the set consisting of all possible re-orders of $\mathbf{s}_i$.
Here $(a)_j$ means a list of $j$ copies of $a$.
\end{Th}
\begin{proof} 

From now on, we will use the following notation 
\begin{eqnarray*}
\sum_{A}&:=&\sum_{\substack{\forall j\neq i\ d_{s_j,1}>\cdots>d_{s_j,k_j}>0\\ d_{s_i,1}>\cdots>d_{s_i,k_i-2}>0\\  \forall(r,e), (t,m)\neq (i,k_i-1),(i,k_i),\ d_{s_r,e}\neq d_{s_t,m}}},\ \sum_{B}:=\sum_{\substack{\forall j\neq i\ d_{s_j,1}>\cdots>d_{s_j,k_j}>0\\ d_{s_i,1}>\cdots>d_{s_i,k_i-1}>0\\ d_{s_a,1}>\cdots>d_{s_a,k_a-1}>0\\ \forall(r,e), (t,m)\neq (i,k_i),(a,k_a),\ d_{s_r,e}\neq d_{s_t,m}}}\\
\sum_{C}&:=&\sum_{\substack{\forall j\neq i\ d_{s_j,1}>\cdots>d_{s_j,k_j}>0\\ d_{s_i,1}>\cdots>d_{s_i,k_i-1}>0\\  \forall(r,e), (t,m)\neq (i,k_i),\ d_{s_r,e}\neq d_{s_t,m}}}\\
\prod&:=&\prod_{j=1}^n\prod_{1\leq l\leq  k_j}\widetilde{S}_{d_{s_j,l}}(s_j),\ \ \ \  {\prod}_i:=\prod_{j\neq i}\prod_{1\leq l\leq k_j}\widetilde{S}_{d_{s_j,l}}(s_j)
\end{eqnarray*}
{\bf{We remark here that when $k_i=2$, the condition $d_{s_i,1}>\cdots>d_{s_i,k_i-2}>0$ from the inner summation should be dropped. From now on, if some line or summation should be dropped because of such kind of issue, we no longer point it out.}}
Consider $i$, such that $k_i>1$. Then we have
\begin{eqnarray*}
\sum_{\sigma\in \textbf{L}_i}\zeta_v(\sigma(\mathbf{s}_i))&=&\sum_{\substack{\forall j\neq i\  d_{s_j,1}>\cdots>d_{s_j,k_j}\geq 0\\ d_{s_i,1}>\cdots>d_{s_i,k_i-1}\geq 0\\\forall(r,e), (t,m)\neq (i,k_i),\ d_{s_r,e}\neq d_{s_t,m}\\ \exists l<k_i\ d_{s_i,l}=d_{s_i,k_i}}}\prod\\
&=&\sum_{\substack{C\\ \exists l<k_i\ d_{s_i,l}=d_{s_i,k_i}}}\prod+\sum_{\substack{\forall j\neq i\  d_{s_j,1}>\cdots>d_{s_j,k_j}\geq 0\\ d_{s_i,1}>\cdots>d_{s_i,k_i-1}\geq 0\\\forall(r,e), (t,m)\neq (i,k_i),\ d_{s_r,e}\neq d_{s_t,m}\\ \exists l<k_i\ d_{s_i,l}=d_{s_i,k_i}\\ \exists a\neq i\ d_{s_a,k_a}=0,\ or\ d_{s_i,k_i-1}=0}}\prod\\
&=&{\sum}_1+{\sum}_2
\end{eqnarray*}

The first equality follows by definition using $S_d(2s_i)=S_d(s_i)^2$ leading to two $d$'s corresponding to $\ell$ and $k_i$, and the second equality follows by distributing the summation into two parts $d_{s_r,k_r}>0,\forall r$; and $\exists a, d_{s_a,k_a}=0$. 

Now we study the second summation first.

We break ${\sum}_2$ to three cases $d_{s_i,k_i-1}=0, d_{s_i,k_i}\neq 0$; $d_{s_i,k_i-1}=d_{s_i,k_i}=0$; and $\exists a\neq i, d_{s_a,k_a}=0$. And applying  $\widetilde{S}_0(s)=\sum_{d=1}^{\infty}\widetilde{S}_d(s)$ in characteristic $2$ for $q-$even $s$, we have the following equality
\begin{eqnarray}
{\sum}_2&=&\sum_{\substack{A\\ \exists l<k_i-1\ d_{s_i,l}=d_{s_i,k_i}}}{\prod}_i\left(\prod_{l\neq k_i-1}\widetilde{S}_{d_{s_i,l}}(s_i)\right)\sum_{d_{s_i,k_i-1}=1}^{\infty}\widetilde{S}_{d_{s_i,k_i-1}}(s_i)\label{E1}\\
&+&\sum_{\substack{A}}{\prod}_i\left(\prod_{l<k_i-1}\widetilde{S}_{d_{s_i,l}}(s_i)\right)\left(\sum_{d_{s_i,k_i}=1}^{\infty}\widetilde{S}_{d_{s_i,k_i}}(2s_i)\right)\label{E2}\\
&+&\sum_{\substack{B\\ \exists l<k_i\ d_{s_i,l}=d_{s_i,k_i}}}{\prod}_a\left(\prod_{h<k_a}\widetilde{S}_{d_{s_a,h}}(s_a)\right)\left(\sum_{d_{s_a,k_a}=1}^{\infty}\widetilde{S}_{d_{s_a,k_a}}(s_a)\right)\label{E3}
\end{eqnarray}

Calculating the above three summations by distributing into cases, we have
\begin{eqnarray*}
(\ref{E1})&=&\sum_{\substack{A\\
\exists l<k_i-1,\ d_{s_i,l}=d_{s_i,k_i}\\ \forall (s_r,w)\neq (s_i,k_i-1)\ d_{s_r,w}\neq d_{s_i,k_i-1}>0}}\prod+\sum_{\substack{A\\
\exists l<k_i-1,\ d_{s_i,k_i}=d_{s_i,l}\\ \exists w\neq l,k_i,\ d_{s_i,w}=d_{s_i,k_i-1}>0}}\prod\\
&+&\sum_{\substack{A\\
\exists l<k_i-1 \ d_{s_i,l}=d_{s_i,k_i-1}=d_{s_i,k_i}>0}}\prod+\sum_{\substack{A\\ \exists l<k_i-1,\ d_{s_i,k_i}=d_{s_i,l}\\
\exists a\neq i,\exists w, \ d_{s_a,w}=d_{s_i,k_i-1}>0}}\prod
\end{eqnarray*}

This equality holds by distributing (\ref{E1}) into four cases, where the typical term corresponding to
$d_{s_i,k_i-1}$ in the last summation:
(I) doesn't agree with any other earlier $d$'s or (II) agrees with some $d_{s_i,w}$ where $w\neq l,k_i$ or (III) agrees with $d_{s_i,l}=d_{s_i,k_i}$ or (IV) agrees with $d_{s_a,w}$ for some $a\neq i$.
The second sum in $(\ref{E1})$ is $0$ because  $d_{s_i,k_i}=d_{s_i,l}, d_{s_i,k_i-1}=d_{s_i,w}$ and $d_{s_i,k_i}=d_{s_i,w}, d_{s_i,k_i-1}=d_{s_i,l}$ will give you same terms since $d_{s_i,k_i}, d_{s_i,k_i-1}$ are unfixed.

Similar distributions give you (\ref{E2}) and (\ref{E3}), where in (\ref{E2}) we again decompose $\widetilde{S}_{d_{s_i,k_i}}(2s_i)=\widetilde{S}_{d_{s_i,k_i}}(s_i)\widetilde{S}_{d_{s_i,k_i-1}}(s_i)$.
\begin{eqnarray*}
(\ref{E2})&=&\sum_{\substack{A\\ d_{s_i,k_i}=d_{s_i,k_i-1}\\ d_{s_r,w}\neq d_{s_i,k_i}}}\prod+\sum_{\substack{A\\ \exists l<k_i-1,\ d_{s_i,l}=d_{s_i,k_i}=d_{s_i,k_i-1}}}\prod\\
&+&\sum_{\substack{A\\
\exists a\neq i, \exists w,\ d_{s_a,w}=d_{s_i,k_i-1}=d_{s_i,k_i}}}\prod\\
(\ref{E3})&=&\sum_{\substack{A\\ \exists l<k_i,\ d_{s_i,k_i}=d_{s_i,l}\\ \forall (s_r,w)\neq (s_a,k_a),\ d_{s_r,w}\neq d_{s_a,k_a}}}\prod+\sum_{\substack{B\\ \exists l<k_i,\ d_{s_i,k_i}=d_{s_i,l}\\ \exists a\neq i, \exists w\neq l,k_i,\  d_{s_a,k_a}=d_{s_i,w}}}\prod\\
&+&\sum_{\substack{B\\ \exists l<k_i,\ d_{s_i,k_i}=d_{s_i,l}=d_{s_a,k_a}}}\prod+\sum_{\substack{B\\ \exists l<k_i,\ d_{s_i,k_i}=d_{s_i,l}\\ \exists b\neq i,a, \exists w,\  d_{s_a,k_a}=d_{s_b,w}}}\prod
\end{eqnarray*}

The fourth sum in $(\ref{E3})$ is $0$, since for each term corresponding to $d_{s_a,k_a}=d_{s_b,w}$, there exists $m$ such that $d_{s_b,k_b}=d_{s_a,m}$ giving rise to the same term. 

We now notice that the third sum in $(\ref{E1})$ is equal to the second sum in $(\ref{E2})$ because of the hypothesis of distinctness three $d$'s can not be $0$ simultaneously. 

The fourth sum in $(\ref{E1})$ is equal to the second sum in $(\ref{E3})$ because for each term in (\ref{E3}) with $d_{s_i,w}=d_{s_a,k_a}$, we can exchange $w$ with $k_i-1$, so this term is transformed to the term with $d_{s_a,k_a}=d_{s_i,k_i-1}$ and we can re-order $(1,\cdots,w-1,k_i-1,w+1\cdots,k_i-2,w)$ by $(1,\cdots,w-1,k_i-1,w,w+1,\cdots,k_i-2)$. If we further ask such term has $d_{s_i,k_i}=d_{s_i,l}\neq d_{s_i,k_i-1}$, then we get $l<k_i-1$. Now if $d_{s_a,w-1}>d_{s_a,k_a}>d_{s_a,w}$ for some $w$, then exchanging $k_a$ with $w$, and after re-ordering $(1,\cdots,w-1,k_a,w+1,\cdots,k_a-1,w)$ by $(1,\cdots,w-1,k_a,w,\cdots,k_a-1)$, terms in second sum in ({\ref{E3}}) match with terms in fourth sum in (\ref{E1}). 

The third sum in $(\ref{E2})$ and $(\ref{E3})$ are also equal because for each term in (\ref{E3}) with $d_{s_i,k_i}=d_{s_i,l}=d_{s_a,k_a}$, we can exchange $l$ with $k_i-1$, and re-ordering $(1,\cdots,l-1,k_i-1,l+1\cdots,k_i-2,l)$ by $(1,\cdots,l-1,k_i-1,l,l+1,\cdots,k_i-2)$. Then, we exchange $k_a$ with $w$, and after re-ordering $(1,\cdots,w-1,k_a,w+1,\cdots,k_a-1,w)$ by $(1,\cdots,w-1,k_a,w,\cdots,k_a-1)$, terms in third sum in ({\ref{E3}}) match with terms in third sum in (\ref{E1}).

Then, we sum up and get 
\begin{eqnarray*}
{\sum}_2&=&\sum_{\substack{A\\
\exists l<k_i-1,\ d_{s_i,l}=d_{s_i,k_i}\\ \forall (s_r,w)\neq (s_i,k_i-1),\ d_{s_r,w}\neq d_{s_i,k_i-1}>0}}\prod+\sum_{\substack{A\\ d_{s_i,k_i}=d_{s_i,k_i-1}>0\\ \forall (s_r,w)\neq (s_i,k_i-1),\ d_{s_r,w}\neq d_{s_i,k_i-1}}}\prod\\
&+&\sum_{\substack{B\\ \exists l<k_i,\ d_{s_i,k_i}=d_{s_i,l}\\ \forall (s_r,w)\neq (s_a,k_a),\ d_{s_r,w}\neq d_{s_a,k_a}}}\prod\\
&=&\left(k_i-2\right)\sum_{\substack{C\\
\exists l<k_i-1,\ d_{s_i,k_l}=d_{s_i,k_i}}}\prod+\sum_{\substack{C\\ \exists l<k_1,\ d_{s_i,k_i}=d_{s_i,k_l}>0}}\prod\\
&+&\left(\sum_{j\neq i}k_j\right)\sum_{\substack{C\\ \exists l<k_i,\ d_{s_i,k_i}=d_{s_i,l}}}\prod\\
&=&\left(1+\sum_{j=1}^n(k_j-2)\right){\sum}_1=(\phi+1){\sum}_1
\end{eqnarray*}

For all $l\neq r<k_i-1$ and $w<k_a$, if we re-order $(1,\cdots,\hat{l},\cdots,r,k_i-1,r+1,\cdots,k_i-2)$, $(1,\cdots,w,k_a,w+1,\cdots,k_a-1)$ by $(1,\cdots,\hat{l},\cdots,k_i-1)$ and $(1,\cdots,k_a)$ respectively, we can get that the first sum and third sum after the first equality sign have $k_i-2$ and $\sum_{j\neq i}k_j$ copies of each term in ${\sum}_1$ respectively. (Here $
\hat{l}$ means we omit the $l$, since $d_{s_i,l}=d_{s_i,k_i}$) And we can directly see that the second sum is one copy of ${\sum}_1$. Then, the second equality holds. The last equality follows because $\sum_{i=1}^n(k_i-2)=\phi-2n=\phi$ in characteristic $2$.

Hence, we sum up for all $i$'s we get 
\begin{equation}
\sum_{i,\ where\ k_i>1}\left(\sum_{\sigma\in \textbf{L}_i}\zeta_v(\sigma(\mathbf{s}_i))\right)=\phi\sum_{i,\ where\ k_i>1}{\sum}_1\label{E4}
\end{equation}

Now we study the ${\sum}_1$ term carefully.
\begin{eqnarray*}
{\sum}_1&=&\sum_{C}{\prod}_i\left(\prod_{l\neq k_i}\widetilde{S}_{d_{s_i,l}}(s_i)\right)\left(\sum_{d_{s_i,k_i}=0}^{\infty}\widetilde{S}_{d_{s_i,k_i}}(s_i)\right)\\
&-&\sum_{\substack{C\\d_{s_r,e}\neq d_{s_i,k_i}\geq 0}}\prod-\sum_{\substack{C\\ \exists a\neq i, \exists h,\ d_{s_a,h}=d_{s_i,k_i}}}\prod\\
&=&\sum_{\substack{C\\ d_{s_r,e}\neq d_{s_i,k_i}\\ d_{s_i,k_i}>0}}\prod+\sum_{\substack{C\\ d_{s_r,e}\neq d_{s_i,k_i}\\ d_{s_i,k_i}=0}}\prod+{\sum}_3\\
&=&k_i\sum_{\substack{d_{s_j,1}>\cdots>d_{s_j,k_j}> 0\\  d_{s_r,e}\neq d_{s_t,m}}}\prod+\sum_{\substack{\forall j\neq i\ d_{s_j,1}>\cdots>d_{s_j,k_j}> 0\\ d_{s_i,1}>\cdots>d_{s_i,k_i}=0\\  d_{s_r,e}\neq d_{s_t,m}}}\prod+{\sum}_3
\end{eqnarray*}

The first equality is by just distributing the product into three cases similar as before. The second equality holds from the fact that $\sum_{d=0}^{\infty}\widetilde{S}_d(s)=0$ in characteristic $2$ for $q-$even $s$ and distributing the second sum into two cases: $d_{s_i,k_i}> 0$ and $d_{s_i,k_i}=0$.

Since we can re-order $(1,\cdots,l-1,k_i,l,\cdots,k_i-1)$ to $(1,\cdots,k_i-1)$ in the first sum for all $l\leq k_i$. Hence, we get $k_i$ copies of the first sum after the last equality sign from the first sum and we have $d_{s_i,k_i}=0$ is obviously the least. Hence, the last equality holds.

We notice that in characteristic $2$, for $i$ such that $k_i=1$, if it exists, we have $\sum_{d_{s_i,1}=0}^{\infty}\widetilde{S}_{d_{s_i,1}}(s_i)=0$ for $s_i$ $q-$even, and
\begin{eqnarray*}
0&=&\sum_{\substack{\forall j\neq u\ d_{s_j,1}>\cdots>d_{s_j,k_j}>0\\ d_{s_r,e}\neq d_{s_t,m}}}{\prod}_u\left(\sum_{d_{s_i,1}=0}^{\infty}\widetilde{S}_{d_{s_i,1}}(s_u)\right)\\
&=&\sum_{\substack{\forall j\neq i\ d_{s_j,1}>\cdots>d_{s_j,k_j}> 0\\\forall r,t\neq i,\ d_{s_r,e}\neq d_{s_t,m}\\ \exists a\neq i,\exists h,\ d_{s_a,h}=d_{s_i,1}}}\prod+\sum_{\substack{\forall j\neq i\  d_{s_j,1}>\cdots>d_{s_j,k_j}> 0\\ d_{s_i,1}=0\\ d_{s_r,e}\neq d_{s_t,m}}}\prod+\sum_{\substack{d_{s_j,1}>\cdots>d_{s_j,k_j}> 0\\ d_{s_r,e}\neq d_{s_t,m}}}\prod
\end{eqnarray*}

Hence, when we sum expression above for ${\sum}_1$ over $i$ where $k_i>1$ and the expression for $0$ above over $i$ where $k_i=1$, we get
\begin{eqnarray*}
&&\sum_{i,\ where\ k_i>1}{\sum}_1\\
&=&\sum_{i,\ where\ k_i>1}{\sum}_1+\sum_{i,\ where\ k_i=1}0\\
&=&\sum_{i,\ where\ k_i>1}\sum_{\substack{\forall j\neq i\  d_{s_j,1}>\cdots>d_{s_j,k_j}> 0\\ d_{s_i,1}>\cdots>d_{s_i,k_i}= 0\\d_{s_r,e}\neq d_{s_t,m}}}\prod+\sum_{i,\ where\ k_i=1}\sum_{\substack{\forall j\neq i\  d_{s_j,1}>\cdots>d_{s_j,k_j}> 0\\ d_{s_i,1}=0\\ d_{s_r,e}\neq d_{s_t,m}}}\prod
\end{eqnarray*}
\begin{eqnarray*}
&+&\sum_{\substack{d_{s_j,1}>\cdots>d_{s_j,k_j}> 0\\  d_{s_r,e}\neq d_{s_t,m}}}\prod+\left(-1+\sum_{i,\ where\ k_i>1}k_i\right)\sum_{\substack{d_{s_j,1}>\cdots>d_{s_j,k_j}> 0\\  d_{s_r,e}\neq d_{s_t,m}}}\prod\\
&+&\sum_{i,\ where\ k_i=1}\sum_{\substack{ d_{s_j,1}>\cdots>d_{s_j,k_j}> 0\\ d_{s_r,e}\neq d_{s_t,m}}}\prod+\sum_{i,\ where\ k_i>1}{\sum}_3\\
&+&\sum_{i,\ where\ k_i=1}\sum_{\substack{\forall j\neq i\  d_{s_j,1}>\cdots>d_{s_j,k_j}> 0\\ \forall r,t\neq i\ d_{s_r,e}\neq d_{s_t,m}\\ \exists a\neq i,\exists h \ d_{s_a,h}=d_{s_i,1}}}\prod
\end{eqnarray*}

(For convenience, we have re-organized the terms so that the second, fifth, and seventh sum are coming from the second, third, and first sum in $0$ expression respectively; the first, sixth sum are derived from the second and last sum in ${\sum}_1$ respectively; the third and fourth sum are summed over the first sum in ${\sum}_1$ and also decompose $\phi=1+(\phi-1$).)

Summing the the first and second sum together, we get the first sum below. Summing the fourth and fifth sum together, we get the third sum below. After distributing the remaining sums into two cases $k_a>1$ and $k_a=1$ accordingly, we have
\begin{eqnarray*}
\sum_{i,\ where\ k_i>1}{\sum}_1&=&\sum_{i=1}^n\sum_{\substack{\forall j\neq i\  d_{s_j,1}>\cdots>d_{s_j,k_j}> 0\\ d_{s_i,1}>\cdots>d_{s_i,k_i}= 0\\d_{s_r,e}\neq d_{s_t,m}}}\prod+\sum_{\substack{d_{s_j,1}>\cdots>d_{s_j,k_j}> 0\\  d_{s_r,e}\neq d_{s_t,m}}}\prod\\
&+&\left(\left(\sum_{i=1}^nk_i\right)-1\right)\sum_{\substack{d_{s_j,1}>\cdots>d_{s_j,k_j}> 0\\  d_{s_r,e}\neq d_{s_t,m}}}\prod\\
&+&\sum_{i,\ where\ k_i>1}\sum_{\substack{C\\ \exists a\neq i,\ where\ k_a>1,\\ \exists h, \ d_{s_a,h}=d_{s_i,k_i}}}\prod+\sum_{i,\ where\ k_i>1}\sum_{\substack{C\\ \exists a\neq i,\ where\ k_a=1,\\  d_{s_a,1}=d_{s_i,k_i}}}\prod\\
&+&\sum_{i,\ where\ k_i=1}\sum_{\substack{\forall j\neq i\  d_{s_j,1}>\cdots>d_{s_j,k_j}> 0\\ \forall r,t\neq i\ d_{s_r,e}\neq d_{s_t,m}\\ \exists a\neq i,\ where \ k_a>1, \\ \exists h,\ d_{s_a,h}=d_{s_i,1}}}\prod+\sum_{i,\ where\ k_i=1}\sum_{\substack{\forall j\neq i\  d_{s_j,1}>\cdots>d_{s_j,k_j}> 0\\ \forall r,t\neq i\ d_{s_r,e}\neq d_{s_t,m}\\ \exists a\neq i,\ where \ k_a=1,\\ d_{s_a,1}=d_{s_i,1}}}\prod\\
&=&\sum_{\sigma\in \textbf{L}_0}\zeta_v(\sigma(\mathbf{s}_0))+(\phi-1)\sum_{\substack{d_{s_j,1}>\cdots>d_{s_j,k_j}> 0\\  d_{s_r,e}\neq d_{s_t,m}}}\prod\\
&+&\sum_{i,\ where\ k_i>1}\sum_{\substack{C\\ \exists a\neq i,\ where\ k_a=1,\\  d_{s_a,1}=d_{s_i,k_i}}}\prod+\sum_{i,\ where\ k_i=1}\sum_{\substack{\forall j\neq i\  d_{s_j,1}>\cdots>d_{s_j,k_j}> 0\\ \forall r,t\neq i\ d_{s_r,e}\neq d_{s_t,m}\\ \exists a\neq i,\ where \ k_a>1, \\ \exists h,\ d_{s_a,h}=d_{s_i,1}}}\prod\\
&=&\sum_{\sigma\in \textbf{L}_0}\zeta_v(\sigma(\mathbf{s}_0))+(\phi-1)\sum_{\substack{d_{s_j,1}>\cdots>d_{s_j,k_j}> 0\\  d_{s_r,e}\neq d_{s_t,m}}}\prod+{\sum}_3+{\sum}_4.
\end{eqnarray*}

The addition of the first two sums equals to $\sum_{\sigma\in \textbf{L}_0}\zeta_v(\sigma(\mathbf{s}_0))$ by definition, and we have $\sum_{i=1}^nk_i=\phi$. After exchanging $a,i$ in expression of the fourth sum (note that $k_i$ is without restriction) and exchanging $a,i$ in the seventh sum, we get $2$ copies of each term in these two sum accordingly, hence they both vanish in characteristic $2$.  This gives us the second equality.

Noting from conditions in ${\sum}_3$ and ${\sum}_4$ that all $d$'s we get in the expressions below are greater $0$, we have

\begin{eqnarray*}
{\sum}_3&=&\sum_{i,\ where\ k_i>1}\sum_{a\neq i,\ where\ k_a=1}\sum_{\substack{\forall j\neq a,i\  d_{s_j,1}>\cdots>d_{s_j,k_j}> 0\\ d_{s_a,1}>0\\ d_{s_i,1}>\cdots>d_{s_i,k_i}>0\\ \forall r,t\neq a\ d_{s_r,e}\neq d_{s_t,m}\\ \exists h,\ d_{s_i,h}=d_{s_a,1}}}\prod,\\
{\sum}_4&=&\sum_{i,\ where\ k_i=1}\sum_{a\neq i,\ where\ k_a>1}\sum_{\substack{\forall j\neq i,a\  d_{s_j,1}>\cdots>d_{s_j,k_j}> 0\\ d_{s_i,1}>0\\ d_{s_a,1}>\cdots>d_{s_a,k_a}>0\\ \forall r,t\neq i\ d_{s_r,e}\neq d_{s_t,m}\\ \exists h,\ d_{s_a,h}=d_{s_i,1}}}\prod.
\end{eqnarray*}

Hence, they are equal to each other once we exchange $i,a$, so ${\sum}_3+{\sum}_4=0$.

Now since $\phi(\phi+1)=0$ in characteristic $2$, we have 
\begin{eqnarray*}
\phi\sum_{i,\ where\ k_i>1}{\sum}_1&=&\phi\sum_{\sigma\in \textbf{L}_0}\zeta_v(\sigma(\mathbf{s}_0))+\phi(\phi-1)\sum_{\substack{d_{s_j,1}>\cdots>d_{s_j,k_j}> 0\\  d_{s_r,e}\neq d_{s_t,m}}}\prod\\
&=&\phi\sum_{\sigma\in \textbf{L}_0}\zeta_v(\sigma(\mathbf{s}_0))
\end{eqnarray*}

Hence, by $(\ref{E4})$, we have
\begin{eqnarray*}
\sum_{{\textit{$i$ where $k_i>1$}}}\left(\sum_{\sigma\in \textbf{L}_i}\zeta_v(\sigma(\mathbf{s}_i))\right)+\phi\sum_{\sigma\in \textbf{L}_0}\zeta_v(\sigma(\mathbf{s}_0))=0.
\end{eqnarray*}

If the assumption fails, i.e. $k_i=1$ for all $i\leq n$, then \Cref{Thm3} can either be reduced to \Cref{Thm2} when $\phi$ is odd, or simply we get $0=0$ when $\phi$ even.
\end{proof}

\begin{Rem}
In the special case where $k_i=1$ for all $i$, we have \Cref{Thm2} as already pointed out in the above proof. Hence, \Cref{Thm3} is a generalization of \Cref{Thm2} in characteristic $2$.
\end{Rem}

We can easily apply the above two results to the finite MZVs and get the following results once we realize the fact that $Z(s)=0$ if $s$ is $q-$even.
\begin{Cor}
$K$ be any given function field, $\mathbf{s}=(s_i)_{i=1}^n$, if $s_i\in \mathbb{Z}$ are distinct and $q-$even and $n$ is odd, then we have:
\begin{eqnarray*}
\sum_{\sigma\in S_n}sgn(\sigma)\zeta^{\mathcal{A}}(\sigma(\mathbf{s}))=0
\end{eqnarray*}
where $S_n$ is the symmetry group and sgn is the sign function on $S_n$.
\end{Cor}

\begin{Cor}
$K$ be any given function field, if $\mathrm{char}(\mathbb{F}_q)=2$, $v$ a prime, $n\geq 1$, $s_i$ and $2s_i$ are all distinct $($\textit{where $2s_i$ is included only when $k_i>1$}$)$, $q-$even and $k_i\geq 1$, $\forall i\leq n$, let $\phi:=\sum_{i=1}^nk_i$, and
$\mathbf{s}_0:=((s_1)_{k_1},\cdots,(s_n)_{k_n}).$
If we further have $k_i> 1$, define
$\mathbf{s}_i:=((s_1)_{k_1},\cdots,(s_{i})_{k_i-2},2s_i,\cdots,(s_n)_{k_n}).$
Then the following holds:
\begin{eqnarray*}
\sum_{{\textit{$i$ where $k_i>1$}}}\left(\sum_{\sigma\in \textbf{L}_i}\zeta^{\mathcal{A}}(\sigma(\mathbf{s}_i))\right)+\phi\sum_{\sigma\in \textbf{L}_0}\zeta^{\mathcal{A}}(\sigma(\mathbf{s}_0))=0.
\end{eqnarray*}
where $\textbf{L}_i$ is the set consisting of all possible re-orders of $\mathbf{s}_i$.
Here $(a)_j$ means a list of $j$ copies of $a$.
\end{Cor}

We can further show that above results hold for $v-$adic and finite MZSVs.
\begin{Th}
$K$ is a function field over $\mathbb{F}_q$ with a rational infinity place, $v$ is a prime, $\mathbf{s}=(s_i)_{i=1}^n$. 
If $s_i\in \mathbb{Z}$ are distinct and $q-$even and $n$ is odd, then we have:
\begin{eqnarray*}
&&\sum_{\sigma\in S_n}sgn(\sigma)\zeta_v^{\star}(\sigma(\mathbf{s}))=0\\
&&\sum_{\sigma\in S_n}sgn(\sigma)\zeta^{\mathcal{A}\star}(\sigma(\mathbf{s}))=0
\end{eqnarray*}
where $S_n$ is the symmetry group and sgn is the sign function on $S_n$.
\end{Th}
\begin{proof}
We only prove for multiple zeta star values, finite multiple zeta values will follow similarly.

We distribute the $v-$adic MZSVs into $v-$adic MZVs and the terms with some $d$'s are equal.
\begin{eqnarray*}
&&\sum_{\sigma\in S_n}sgn(\sigma)\zeta_v^{\star}(\sigma(\mathbf{s}))\\
&=&\sum_{\sigma\in S_n}sgn(\sigma)\zeta_v(\sigma(\mathbf{s}))+\sum_{\substack{\sigma\in S_n\\ d_{s_{\sigma(1)}}\geq \cdots\geq d_{s_{\sigma(n)}}\geq 0\\ \exists i<j,\ d_{s_{\sigma(i)}}=d_{s_{\sigma(j)}}}}sgn(\sigma)\prod_{i=1}^n\widetilde{S}_{d_{s_{\sigma(i)}}}(s_{\sigma(i)})\\
&=&\sum_{\substack{\exists i<j,\ d_{s_{\sigma(i)}}=d_{s_{\sigma(j)}}\\ \sigma\in S_n, \textit{ where $\sigma(i)<\sigma(j)$}\\ d_{s_{\sigma(1)}}\geq \cdots\geq d_{s_{\sigma(n)}}\geq 0}}sgn(\sigma)\prod_{i=1}^n\widetilde{S}_{d_{s_{\sigma(i)}}}(s_{\sigma(i)})\\
&+&\sum_{\substack{\exists i<j,\ d_{s_{\sigma(i)}}=d_{s_{\sigma(j)}}\\ \sigma\in S_n, \textit{ where $\sigma(i)>\sigma(j)$}\\ d_{s_{\sigma(1)}}\geq \cdots\geq d_{s_{\sigma(n)}}\geq 0}}sgn(\sigma)\prod_{i=1}^n\widetilde{S}_{d_{s_{\sigma(i)}}}(s_{\sigma(i)})\\
&=&\sum_{\substack{\exists i<j,\ d_{s_{\sigma(i)}}=d_{s_{\sigma(j)}}\\ \sigma\in S_n, \textit{ where $\sigma(i)<\sigma(j)$}\\ d_{s_{\sigma(1)}}\geq \cdots\geq d_{s_{\sigma(n)}}\geq 0}}\left(sgn(\sigma)+sgn(i,j)sgn(\sigma)\right)\prod_{i=1}^n\widetilde{S}_{d_{s_{\sigma(i)}}}(s_{\sigma(i)})\\
&=&0,
\end{eqnarray*}
where the second equality follows by \Cref{Thm2}.
\end{proof}

\begin{Th}
Let $K$ be a function fields with a rational infinity. If $\mathrm{char}(\mathbb{F}_q)=2$, $v$ a monic prime, $n\geq 1$, $s_i$ and $2s_i$ are all distinct $($\textit{where $2s_i$ is included only when $k_i>1$}$)$, $q-$even and $k_i\geq 1$, $\forall i\leq n$, let $\phi:=\sum_{i=1}^nk_i$, and
$\mathbf{s}_0:=((s_1)_{k_1},\cdots,(s_n)_{k_n}).$
If we further have $k_i> 1$, define
$$\mathbf{s}_i:=((s_1)_{k_1},\cdots,(s_{i})_{k_i-2},2s_i,\cdots,(s_n)_{k_n}).$$
Then the following holds:
\begin{eqnarray*}
\sum_{{\textit{$i$ where $k_i>1$}}}\left(\sum_{\sigma\in \textbf{L}_i}\zeta^{\star}_v(\sigma(\mathbf{s}_i))\right)+\phi\sum_{\sigma\in \textbf{L}_0}\zeta^{\star}_v(\sigma(\mathbf{s}_0))=0.
\end{eqnarray*}
\begin{eqnarray*}
\sum_{{\textit{$i$ where $k_i>1$}}}\left(\sum_{\sigma\in \textbf{L}_i}\zeta^{\mathcal{A}\star}(\sigma(\mathbf{s}_i))\right)+\phi\sum_{\sigma\in \textbf{L}_0}\zeta^{\mathcal{A}\star}(\sigma(\mathbf{s}_0))=0.
\end{eqnarray*}
where $\textbf{L}_i$ is the set consisting of all possible re-orders of $\mathbf{s}_i$.
Here $(a)_j$ means a list of $j$ copies of $a$.
\end{Th}
\begin{proof}

We again only prove for multiple zeta star values, finite multiple zeta values will follow similarly.

Distribute and reordering the following two sums by the fact that $\widetilde{S}_d^2(s)=\widetilde{S}_d(2s)$ in characteristic $2$ and $\zeta_v(s)=0$ when $s$ is $q-$even.

When $k_i>1$, we have
\begin{eqnarray*}
&&\sum_{\sigma\in \textbf{L}_i}\zeta^{\star}_v(\sigma(\mathbf{s}_i))+k_i\sum_{\sigma\in \textbf{L}_0}\zeta^{\star}_v(\sigma(\mathbf{s}_0))\\
&=&\sum_{\substack{\forall j\neq i\  d_{s_j,1}\geq \cdots\geq d_{s_j,k_j}\geq 0\\ d_{s_i,1}\geq \cdots\geq d_{s_i,k_i-1}\geq 0\\ \exists l<k_i\ d_{s_i,l}=d_{s_i,k_i}}}\prod+\sum_{d_{s_j,1}\geq \cdots\geq d_{s_j,k_j}\geq 0}\prod+\sum_{l=1}^{k_i-1}\sum_{\substack{\forall j\neq i\  d_{s_j,1}\geq \cdots\geq d_{s_j,k_j}\geq 0\\ d_{s_i,1}\geq \cdots\geq d_{s_i,k_i-1}\geq 0\\ d_{s_i,k_i}\geq d_{s_i,l}}}\prod\\
&=&\sum_{\substack{\forall i\neq j,\ d_{s_j,1}\geq \cdots\geq d_{s_j,k_j}\geq 0\\ d_{s_i,1}\geq \cdots\geq d_{s_i,k_i-1}\geq 0}}{\prod}_i\prod_{h<k_i}\widetilde{S}_{d_{s_i,h}}(s_i)\left(\sum_{d_{s_i,k_i}}\widetilde{S}_{d_{s_i,k_i}}(s_i)\right)\\
&=&\sum_{\sigma\in \textbf{L}_i^i}\zeta^{\star}_v(\sigma(\mathbf{s}_0^i))\zeta_v(s_i)=0,
\end{eqnarray*}

when $k_i=1$, we have
\begin{eqnarray*}
\sum_{\sigma\in \textbf{L}_0}\zeta^{\star}_v(\sigma(\mathbf{s}_0))&=&\sum_{d_{s_j,1}\geq \cdots\geq d_{s_j,k_j}\geq 0}\prod\\
&=&\sum_{\substack{\forall i\neq j,\ d_{s_j,1}\geq \cdots\geq d_{s_j,k_j}\geq 0\\ d_{s_i,1}\geq \cdots\geq d_{s_i,k_i-1}\geq 0}}{\prod}_i\prod_{h<k_i}\widetilde{S}_{d_{s_i,h}}(s_i)\left(\sum_{d_{s_i,k_i}}\widetilde{S}_{d_{s_i,k_i}}(s_i)\right)\\
&=&\sum_{\sigma\in \textbf{L}_i^i}\zeta^{\star}_v(\sigma(\mathbf{s}_0^i))\zeta_v(s_i)=0,
\end{eqnarray*}

where $$\mathbf{s}_0^i:=((s_1)_{k_1},\cdots,(s_i)_{k_i-1},\cdots,(s_n)_{k_n}).$$
Hence, we proved this theorem.
\end{proof}

In Thakur's paper \cite[5.4]{T09}, he conjectured that there is no $\mathbb{F}_q-$linear relation between MZVs.
\begin{Conj}[Thakur]\label{lin.ind.conj}
Given a rational function field $K=\mathbb{F}_q(t)$. The MZVs $\zeta(\mathbf{s})$ at positive integers are linearly independent over $\mathbb{F}_q$.
\end{Conj}
In the case of $v-$adic MZVs we do find some nice relations of MZVs over $\mathbb{F}_q$. When $q=2,3,4,5,7$, weight smaller than $10,18,24,24,30$ respectively, we calculated the $\mathbb{F}_q-$linear relations for all $t-$adic MZVs with $q-$even tuples. And we experiment whether these relations work for higher degree primes (we only check for degree $\leq 5$). 

Based on the data, and the above results, We have the following slightly different conjectures:
\begin{Def}
Given any rational function field $K=\mathbb{F}_q(t)$ and any monic prime $v$, We define $W_v(K):=$ the vector space consisting of all $v-$adic MZVs $\zeta_v(\mathbf{s})$ at integers over $K$. And we define 
$$W(K):=(W_v(K))_{v \textit{ prime}}=\prod_{v \textit{ prime}}W_v(K).$$
\end{Def}

\begin{Conj}\label{lin.conj}
Given a rational function field $K=\mathbb{F}_q(t)$. If $\mathrm{char}(\mathbb{F}_q)=2$, the linear relations over $\mathbb{F}_q$ in $W(K)$ are spanned by \Cref{Thm2} and trivial zeros.
\end{Conj}

\begin{Conj}\label{lin.conj2}
Given a rational function field $K=\mathbb{F}_q(t)$ and a finite monic prime $v$. If $\mathrm{char}(k)\neq 2$, the linear relations over $\mathbb{F}_q$ in $W_v(K)$ are spanned by \Cref{Thm3} and trivial zeros.
\end{Conj}
\begin{Rem}
We notice that the definition of trivial zeros varies when $v$ changes, but it's easy to see that the intersection of all trivial zeros corresponding to different finite primes is nonempty, and we refer to this intersection as our definition of trivial zeros in above conjectures. We have proved in \cite{Shen19} that these zero sets are nonempty.
\end{Rem}

\section{The MZVs Case and Conjectural Motivic Case}
We have proved some universal $\mathbb{F}_q-$linear relations in the case of interpolated $v-$adic MZVs. Similarly, we have a motivic version of \Cref{Thm2}, and \Cref{Thm3}. We won't give the construction of motivic MZVs here, one can read \cite{CM17} for the set-up and details.

\begin{Th}\label{A}
Let $K$ be any function field, $\mathbf{s}=(s_i)_{i=1}^n$, $\mathbf{s}^j:=(s_i)_{i\neq j}$. 
If $s_i$ are distinct and $q-$even and $n$ is odd, then we have:
\begin{eqnarray*}
\sum_{\sigma\in S_n}sgn(\sigma)\zeta(\sigma(\mathbf{s}))=\sum_{j=1}^n(-1)^{n-j}\zeta(s_j)\sum_{\sigma\in S_n^j}sgn(\sigma)\zeta(\sigma(\mathbf{s}^j))
\end{eqnarray*}
where $S_n, S_n^j$ are the symmetry group of $\{1,\cdots,n\}$, and $\{1,\cdots,\hat{j},\cdots,n\}$ respectively, in particular, we have $S_n^j\cong S_{n-1}$, sgn is the sign function on $S_n$.
\end{Th}
\begin{proof}
We first notice that there exists a natural embedding $\iota_j: S_n^j\longrightarrow S_n$, such that $\forall \sigma\in S_n^j$:
$$\iota_j(\sigma)(j)=j,\ \forall k\neq j,\ \iota_j(\sigma)(k)=\sigma(k)$$
and $sgn(\iota_j(\sigma))=sgn(\sigma)$.
Let $\prod:=S_{d_{s_1}}(s_1)\cdots S_{d_{s_n}}(s_n).$
Then, we have 

\begin{eqnarray*}
&&\sum_{j=1}^n(-1)^{n-j}\zeta(s_j)\sum_{\sigma\in S_n^j}sgn(\sigma)\zeta(\sigma(\mathbf{s}^j))\\
&=&\sum_{j=1}^n(-1)^{n-j}\sum_{\substack{\sigma\in S_n^j\\ d_{s_{\sigma(1)}>\cdots>d_{s_{\sigma(j-1)}}>d_{s_{\sigma(j+1)}}>\cdots>d_{s_{\sigma(n)}}\geq 0}\\ d_{s_j}\geq 0}}sgn(\sigma)\prod\\
&=&\sum_{j=1}^n(-1)^{n-j}\sum_{\substack{\sigma\in S_n^j\\ d_{s_{\sigma(1)}>\cdots>d_{s_{\sigma(j-1)}}>d_{s_{\sigma(j+1)}}>\cdots>d_{s_{\sigma(n)}}\geq 0}\\ \exists k\neq j,\ d_{s_j}=d_{s_k}}}sgn(\sigma)\prod\\
&+&\sum_{j=1}^n(-1)^{n-j}\sum_{\substack{\sigma\in S_n^j\\ d_{s_{\sigma(1)}>\cdots>d_{s_{\sigma(j-1)}}>d_{s_{\sigma(j+1)}}>\cdots>d_{s_{\sigma(n)}}\geq 0}\\ \forall k\neq j,\ d_{s_j}\neq d_{s_k}\geq 0}}sgn(\sigma)\prod\\
&=&{\sum}_1+{\sum}_2
\end{eqnarray*}

Now we consider the sums in ${\sum}_1$ with $k<j$.
For all terms with $\sigma\in S_n^j, d_{s_{\sigma(1)}}>\cdots>d_{s_{\sigma(j-1)}}>d_{s_{\sigma(j+1)}}>\cdots>d_{s_{\sigma(n)}}\geq 0, d_{s_j}=d_{s_k}$. 

If $\sigma^{-1}(k)<k<j$, then $\exists \widetilde{\sigma}\in S_n^k$ such that $$\widetilde{\sigma}=\sigma\circ\begin{pmatrix}
    \sigma^{-1}(k) & \cdots & k \\
    \sigma^{-1}(k)+1 & \cdots &  \sigma^{-1}(k)
  \end{pmatrix}\circ\begin{pmatrix}
    \sigma^{-1}(k) & \cdots & \hat{k} & \cdots & j \\
    j & \cdots & \hat{k} & \cdots & j-1
  \end{pmatrix}.$$
It's easy to show that $\widetilde{\sigma}(\sigma^{-1}(k))=j, \widetilde{\sigma}(k)=k$, and give rise to the chain $d_{s_{\sigma(1)}}>\cdots>d_{s_{\sigma^{-1}(k)-1}}>d_{s_j}>d_{s_{\sigma^{-1}(k)+1}}>\cdots>d_{s_{\sigma(j-1)}}>d_{s_{\sigma(j+1)}}>\cdots>d_{s_{\sigma(n)}}\geq 0, d_{s_j}=d_{s_k}$. Then, these two permutations give rise to the same product in ${\sum}_1$, with $(-1)^{n-k}sgn(\widetilde{\sigma})=(-1)^{n-k}(-1)^{k-j-1}sgn(\sigma)=-(-1)^{n-j}sgn(\sigma)$. Hence they add up to $0$.

If $k\leq \sigma^{-1}(k)<j$, then $\exists \widetilde{\sigma}\in S_n^k$ such that $$\widetilde{\sigma}=\sigma\circ\begin{pmatrix}
    k & \cdots & \sigma^{-1}(k) \\
    \sigma^{-1}(k) & \cdots &  \sigma^{-1}(k)-1
  \end{pmatrix}\circ\begin{pmatrix}
    \sigma^{-1}(k)+1 & \cdots  & j \\
    j & \cdots &  j-1
  \end{pmatrix}.$$
Same argument gives us that these two permutations give rise to the same product in ${\sum}_1$, with $(-1)^{n-k}sgn(\widetilde{\sigma})=-(-1)^{n-j}sgn(\sigma)$. Hence they add up to $0$.

If $k<j<\sigma^{-1}(k)$, then $\exists \widetilde{\sigma}\in S_n^k$ such that $$\widetilde{\sigma}=\sigma\circ\begin{pmatrix}
    k & \cdots & \hat{j} & \cdots & \sigma^{-1}(k) \\
    \sigma^{-1}(k) & \cdots & \hat{j} & \cdots & \sigma^{-1}(k)-1
  \end{pmatrix}\circ\begin{pmatrix}
    j & \cdots  & \sigma^{-1}(k) \\
    j+1 & \cdots &  j
  \end{pmatrix}.$$
Same argument gives us that these two permutations give rise to the same product in ${\sum}_1$, with $(-1)^{n-k}sgn(\widetilde{\sigma})=-(-1)^{n-j}sgn(\sigma)$. Hence they add up to $0$.

Adding them all we get ${\sum}_1=0$, (notice that the terms with $j<k$ is added in $\widetilde{\sigma}$, so we didn't miss any terms).

We claim that if we fix an $n-$element set $D:=\{d_{s_i}\}_{i=1}^n$, then each product shows up in ${\sum}_2$ with the $d$'s in $D$ is corresponding to a $\sigma\in S_n$ and actually, for all $\sigma\in S_n$, there exists a unique term in $S_n^j$ corresponding to it for each $j$. Hence it's an $n$ to $1$ correspondence.

Now we prove this claim:
For all $\sigma\in S_n$, we consider the term product from it by $d_{s_{\sigma(1)}}>d_{s_{\sigma(j)}}>\cdots>d_{s_{\sigma(n)}}$
we can always define 
\begin{eqnarray*}
\Bar{\sigma}&:=&\begin{pmatrix}
    \sigma(j) & \cdots  & n \\
    n & \cdots &  n-1
  \end{pmatrix}\circ\sigma\\
&=&\begin{pmatrix}
    1 & \cdots  & \sigma(j)-1 & \sigma(j) & \sigma(j)+1 & \cdots & n \\
    \sigma(1) & \cdots &  \sigma(\sigma(j)-1) & n & \sigma(\sigma(j)+1) & \cdots & \sigma(n)
    \end{pmatrix}
\end{eqnarray*}
We can further define
$$\sigma_{\sigma(j)}:=\begin{pmatrix}
    \sigma(j) & \cdots  & n \\
    \sigma(j)+1 & \cdots &  \sigma(j)
  \end{pmatrix}\circ\Bar{\sigma}\in S_n^{\sigma(j)},$$
and the product corresponding to $\sigma_{\sigma(j)}$ in the sum indexed by $S_n^{\sigma(j)}$ gives rise to such product.

The sign of each term should be $(-1)^{n-\sigma(j)}sgn(\sigma_{\sigma(j)})=(-1)^{n-\sigma(j)}sgn(\sigma)$. Hence, when sum them up with $j$ runs from $1$ to $n$, we get $sgn(\sigma)$ since $n$ is odd. This completes the proof.
\end{proof}

\begin{Th}\label{B}
Let $K$ be any function field, if $char(\mathbb{F}_q)=2$, $n\geq 1$, $s_i$ and $2s_i$ are all distinct(\textit{where $2s_i$
is included only when $k_i>1$}), $q-$even and $k_i\geq 1$, $\forall i\leq n$, let $\phi:=\sum_{i=1}^nk_i$, and
$\mathbf{s}_0:=((s_1)_{k_1},\cdots,(s_n)_{k_n}).$
If we further have $k_i> 1$, define
$\mathbf{s}_i:=((s_1)_{k_1},\cdots,(s_{i})_{k_i-2},2s_i,\cdots,(s_n)_{k_n}).$
Then the following holds:
\begin{eqnarray*}
&&\sum_{{\textit{$i$ where $k_i>1$}}}\sum_{\sigma\in \textbf{L}_i}\zeta(\sigma(\mathbf{s}_i))+\phi\sum_{\sigma\in \textbf{L}_0}\zeta(\sigma(\mathbf{s}_0))\\
&=&\sum_{j, \ where \ k_j\neq 2}\zeta(s_j)\sum_{{\textit{$i\neq j$ where $k_i>1$}}}\sum_{\sigma\in \textbf{L}_i^j}\zeta(\sigma(\mathbf{s}_i^j))\\
&+&\sum_{j, \ where \ k_j> 2}\zeta(s_j)\sum_{\sigma\in \textbf{L}_j^j}\zeta(\sigma(\mathbf{s}_j^j))+\sum_{j,\ where\  k_j>1}\zeta(2s_j)\sum_{\sigma\in \textbf{L}_j^j}\zeta(\sigma(\mathbf{s}_j^{2j}))\\
&+&\sum_{j=1}^n\zeta(s_j)\phi\sum_{\sigma\in \textbf{L}_0^j}\zeta(\sigma(\mathbf{s}_0^j))
\end{eqnarray*}
where $$\mathbf{s}_i^j:=((s_1)_{k_1},\cdots,(s_{i})_{k_i-2},2s_i,\cdots,(s_j)_{k_j-1},\cdots,(s_n)_{k_n}),$$
$$\mathbf{s}_i^i:=((s_1)_{k_1},\cdots,(s_{i})_{k_i-3},2s_i,\cdots,(s_n)_{k_n}),$$
$$\mathbf{s}_i^{2j}:=((s_1)_{k_1},\cdots,(s_{i})_{k_i-2},2s_i,\cdots,(s_j)_{k_j-1},(s_{j+1})_{k_{j+1}}\cdots,(s_n)_{k_n}),$$
and $\textbf{L}_i^j$ is the set consisting of all possible re-orders of $\mathbf{s}_i^j$, when $k_i>1,j\neq i$ or $k_i>2,j=i$.
\end{Th}
\begin{proof}
From now on, we denote 
\begin{eqnarray*}
{\prod}_0^i&:=&\widetilde{S}_{d_{s_i}}(s_i)\prod_{d_{s_{\sigma(1),1}>\cdots>d_{s_{\sigma(1),k_1}}>\cdots>d_{s_{\sigma(n),k_n}}}}\prod_{j\neq i}^n\prod_{l=1}^{k_j}\widetilde{S}_{d_{s_j,l}}(s_j)\prod_{l=1}^{k_i-1}\widetilde{S}_{d_{s_i,l}}(s_i),\\
{\prod}_{2i}^j&:=&\widetilde{S}_{d_{s_j}}(s_j)\prod_{d_{s_{\sigma(1),1}>d_{s_{\sigma(1),k_1}}>\cdots>d_{s_{\sigma(i),k_i-2}}>d_{s_{\sigma(2s_i)}}>\cdots>d_{s_{\sigma(n),k_n}}}}\prod_{r\neq i,j}\prod_{l=1}^{k_r}\widetilde{S}_{d_{s_r,l}}(s_r)\\
&&\prod_{l=1}^{k_j-1}\widetilde{S}_{d_{s_j,l}}(s_j)\prod_{l=1}^{k_i-2}\widetilde{S}_{d_{s_i,l}}(s_i)\widetilde{S}_{d_{2s_i}}(2s_i) \textit{ where $j\neq i$},\\
{\prod}_{2j}^j&:=&\widetilde{S}_{d_{s_j}}(s_j)\prod_{d_{s_{\sigma(1),1}>d_{s_{\sigma(1),k_1}}>\cdots>d_{s_{\sigma(i),k_i-2}}>d_{s_{\sigma(2s_i)}}>\cdots>d_{s_{\sigma(n),k_n}}}}\prod_{r\neq j}\prod_{l=1}^{k_r}\widetilde{S}_{d_{s_r,l}}(s_r)\\
&&\prod_{l=1}^{k_j-3}\widetilde{S}_{d_{s_j,l}}(s_j)\widetilde{S}_{d_{2s_j}}(2s_j),\\
{\prod}_{2j}^{2j}&:=&\widetilde{S}_{d_{2s_j}}(2s_j)\prod_{d_{s_{\sigma(1),1}>d_{s_{\sigma(1),k_1}}>\cdots>d_{s_{\sigma(i),k_i-2}}>d_{s_{\sigma(2s_i)}}>\cdots>d_{s_{\sigma(n),k_n}}}}\prod_{r\neq j}\prod_{l=1}^{k_r}\widetilde{S}_{d_{s_r,l}}(s_r)\\
&&\prod_{l=1}^{k_j-2}\widetilde{S}_{d_{s_j,l}}(s_j).
\end{eqnarray*}

Then, we consider each sum in above expression separately, and we have
\begin{eqnarray*}
&&\sum_{j, \ where \ k_j\neq 2}\zeta(s_j)\sum_{{\textit{$i\neq j$ where \ $k_i>1$}}}\sum_{\sigma\in \textbf{L}_i^j}\zeta(\sigma(\mathbf{s}_i^j))\\
&=&\sum_{j, \ where \ k_j\neq 2}k_j\sum_{i\neq j,\ where \ k_i>1}\sum_{\sigma\in \textbf{L}_i}\zeta(\sigma(\mathbf{s}_i))\\
&+&\sum_{j,\ where \ k_j>2}\sum_{i\neq j,\ where \ k_i>2}\sum_{\substack{\sigma\in \textbf{L}_i^j\\ \exists l\leq k_j-1,\ d_{s_j}=d_{s_{j,l}}}}{\prod}_{2i}^j\\
&+&\sum_{j, \ where \ k_j\neq 2}\sum_{i\neq j,\ where \ k_i>1}\sum_{\substack{\sigma\in \textbf{L}_i^j\\ d_{s_j}=d_{2s_i}}}{\prod}_{2i}^j\\
&+&\sum_{j, \ where \ k_j\neq 2}\sum_{i\neq j,\ where \ k_i>2}\sum_{\substack{\sigma\in \textbf{L}_i^j\\ \exists l\leq k_i-2,\  d_{s_j}=d_{s_{i,l}}}}{\prod}_{2i}^j\\
&+&\sum_{j, \ where \ k_j\neq 2}\sum_{i\neq j,\ where \ k_i>1}\sum_{\substack{\sigma\in \textbf{L}_i^j\\ \exists e\neq i,j,\ \exists l\leq k_e,\  d_{s_j}=d_{s_{e,l}}}}{\prod}_{2i}^j\\
&=&A+B+C+D+E
\end{eqnarray*}

We get this by splitting the sum into different cases, and in the case when each $d$'s are distinct, using the fact that there are $k_j-2$ position to put a $s_{j}$ in a string $(s_{j,1},\cdots,s_{j,k_j-3})$, Hence we get $A$.
We can see that $E=0$ in characteristic $2$ since by exchanging $j,e$, we get $2$ copies of each term in $E$. Similarly, we get $B=0$ by exchanging $d_{s_j}=d_{s_j,l}, d_{2s_j}$ in characteristic $2$.
\begin{eqnarray*}
&&\sum_{j, \ where \ k_j> 2}\zeta(s_j)\sum_{\sigma\in \textbf{L}_j^j}\zeta(\sigma(\mathbf{s}_j^j))\\
&=&\sum_{j, \ where \ k_j> 2}(k_j-2)\sum_{\sigma\in \textbf{L}_j}\zeta(\sigma(\mathbf{s}_j))+\sum_{j, \ where \ k_j> 3}\sum_{\substack{\sigma\in \textbf{L}_j^j\\ \exists l\leq k_j-3,\ d_{s_j}=d_{s_{j,l}}}}{\prod}_{2j}^j\\
&+&\sum_{j, \ where \ k_j> 2}\sum_{\substack{\sigma\in \textbf{L}_j^j\\ \exists l\leq k_j-3,\ d_{s_j}=d_{2s_j}}}{\prod}_{2j}^j+\sum_{j, \ where \ k_j> 2}\sum_{\substack{\sigma\in \textbf{L}_j^j\\ \exists e\neq j,\ \exists l\leq k_e,\ d_{s_j}=d_{s_{e,l}}}}{\prod}_{2j}^j\\
&=&F+G+H+I
\end{eqnarray*}

Similar reason to get this. And we notice that $G=0$ in characteristic $2$ by exchanging $d_{s_j}=d_{s_j,l}, d_{2s_j}$. And we have $I=D$ by exchanging $i,j$ in characteristic $2$.

\begin{eqnarray*}
&&\sum_{j,\ where\  k_j>1}\zeta(2s_j)\sum_{\sigma\in \textbf{L}_j^j}\zeta(\sigma(\mathbf{s}_j^{2j}))\\
&=&\sum_{j,\ where\  k_j>1}\sum_{\sigma\in \textbf{L}_j}\zeta(\sigma(\mathbf{s}_j))+\sum_{j,\ where\  k_j>1}\sum_{\substack{\sigma\in \textbf{L}_j^{2j}\\ \exists l\leq k_j-2,\ d_{2s_j}=d_{s_{j,l}}}}{\prod}_{2j}^{2j}\\
&+&\sum_{j,\ where\  k_j>1}\sum_{\substack{\sigma\in \textbf{L}_j^{2j}\\ \exists e\neq j,\ \exists l\leq k_e,\ d_{2s_j}=d_{s_{e,l}}}}{\prod}_{2j}^{2j}\\
&=&J+K+L
\end{eqnarray*}

Now we have $H=K$, $C=L$ in characteristic $2$.

\begin{eqnarray*}
&&\sum_{j=1}^n\zeta(s_j)\sum_{\sigma\in \textbf{L}_0}\zeta(\sigma(\mathbf{s}_0^j))\\
&=&\sum_{j=1}^nk_j\sum_{\sigma\in \textbf{L}_0^{j}}\zeta(\sigma(\mathbf{s}_0))+\sum_{j, where\  j>1}\sum_{\substack{\sigma\in \textbf{L}_0^j\\ \exists l\leq k_j,\ d_{s_j}=d_{s_{j,l}}}}{\prod}_0^j+\sum_{j=1}^n\sum_{\substack{\sigma\in \textbf{L}_0^{j}\\ \exists e\neq j,\ l\leq k_e,\ d_{s_j}=d_{s_{e,l}}}}{\prod}_0^j\\
&=&M+N+O
\end{eqnarray*}

Here we have $O=0$ in characteristic $2$.
So the remaining terms add up to
\begin{eqnarray*}
&&A+F+J+\phi (N+M)\\
&=&\left(\sum_{r, \ where \ k_r\neq 2}k_r\sum_{j\neq r,\ where \ k_j>1}+\sum_{j,\ where\ k_j>2}k_j-2+\sum_{j,\ where\ k_j>1}+\phi\sum_{j,\ where k_j>1}\right)\\
&&\sum_{\sigma\in \textbf{L}_j}\zeta(\sigma(\mathbf{s}_j))+\phi\sum_{j=1}^n\sum_{\sigma\in \textbf{L}_0}\zeta(\sigma(\mathbf{s}_0))\\
&=&\left(\phi\sum_{j,\ where\ k_j>1}-\sum_{j,\ where k_j>1}k_j+\sum_{j,\ where\ k_j>2}k_j+\sum_{j,\ where\ k_j>1}+\phi\sum_{j,\ where\ k_j>1}\right)\\
&&\sum_{\sigma\in \textbf{L}_j}\zeta(\sigma(\mathbf{s}_j))+\phi\sum_{j=1}^n\sum_{\sigma\in \textbf{L}_0}\zeta(\sigma(\mathbf{s}_0))\\
&=&\sum_{{\textit{$i$ where $k_i>1$}}}\sum_{\sigma\in \textbf{L}_i}\zeta(\sigma(\mathbf{s}_i))+\phi\sum_{\sigma\in \textbf{L}_0}\zeta(\sigma(\mathbf{s}_0))
\end{eqnarray*}
This follows from the fact that we are in characteristic $2$ fields.
\end{proof}
\begin{Rem}
Actually, the proofs works at $d_i$ (unordered) tuple-wise, so works at $S_d$ level, so the above theorems remain true if we replace each $\zeta(\mathbf{s})$ by $\zeta(d,\mathbf{s})$, i.e. the $d-$th truncation of multiple zeta values. 
\end{Rem}

\begin{Def}
Given a rational function field $K=\mathbb{F}_q(t)$ and a monic finite prime $v$, we define

$M:={\normalfont \text{Span}}_{\mathbb{F}_q(t)}\{ \zeta(\mathbf{s}): s_j>0 \}$, 
$\widetilde{M}:={\normalfont \text{Span}}_{\mathbb{F}_q(t)}\{ \zeta_v^{C-M}(\mathbf{s}): s_j>0 \}$ where we use $\zeta_v^{C-M}$ to denote Chang and Mishiba's motivic $v-$adic MZVs.

Define $\phi_v$ to be the map $\phi_v: M\longrightarrow \widetilde{M}$, given by $\phi_v(\zeta(\mathbf{s}))=\zeta_v(\mathbf{s})$.
\end{Def}

\begin{Th}\cite{CM17}
$\phi_v$ defined as above is a well defined surjective $\mathbb{F}_q(t)-$linear map for each prime $v$.
\end{Th}

We know that $M$ has a natural product structure induced from the shuffle product of words.
What's more, lots of calculation evidence suggests that $\phi_v$ may be able to induce an algebraic structure on $\widetilde{M}$ as a push forward of the shuffle algebra structure on $M$. If this is the case, then applying the above results, we get the following results.

\begin{Conj}\label{Thm4}
Let $K=\mathbb{F}_q(t)$, $v$ is a prime, $\mathbf{s}=(s_i)_{i=1}^n$. 
If $s_i$ are distinct and $q-$even and $n$ is odd, then we have:
\begin{eqnarray*}
\sum_{\sigma\in S_n}sgn(\sigma)\zeta_v^{C-M}(\sigma(\mathbf{s}))=0
\end{eqnarray*}
where $S_n$ is the symmetry group and sgn is the sign function on $S_n$.
\end{Conj}

\begin{Conj}\label{Thm5}
Let $K=\mathbb{F}_q(t)$, $\mathrm{char}(\mathbb{F}_q)=2$, $v$ a prime, $n\geq 1$, $s_i$ and $2s_i$ are all distinct $($\textit{where $2s_i$ is included only when $k_i>1$}$)$, $q-$even and $k_i\geq 1$, $\forall i\leq n$. Let $\phi:=\sum_{i=1}^nk_i$, and
$\mathbf{s}_0:=((s_1)_{k_1},\cdots,(s_n)_{k_n}).$
If we further have $k_i> 1$, define
$$\mathbf{s}_i:=((s_1)_{k_1},\cdots,(s_{i})_{k_i-2},2s_i,\cdots,(s_n)_{k_n}).$$
Then the following holds:
\begin{eqnarray*}
\sum_{{\textit{$i$ where $k_i>1$}}}\left(\sum_{\sigma\in \textbf{L}_i}\zeta_v^{C-M}(\sigma(\mathbf{s}_i))\right)+\phi\sum_{\sigma\in \textbf{L}_0}\zeta_v^{C-M}(\sigma(\mathbf{s}_0))=0.
\end{eqnarray*}
where $\textbf{L}_i$ is the set consisting of all possible re-orders of $\mathbf{s}_i$.
Here $(a)_j$ means a list of $j$ copies of $a$.
\end{Conj}

\section{Multiple Harmonic Type Sum}

In this section, we generalize our results to more general cases and find some interesting family of relations of classical MZVs.
\begin{Def}
A sum $H(s_1,\cdots,s_r)$ is called of Multiple Harmonic Type taking values in a commutative ring $(R;``+",``\cdot")$ (infinte formal sum in $R$ is allowed), if 
$$H(s_1,\cdots,s_r):=\sum_{d_1>\cdots>d_r}h(d_1,s_1)\cdots h(d_r,s_r),$$
where each $d_i\in D$, $s_i\in S$, $(D;``>")$ is in a totally ordered set, $(S;``+")$ is a magma, and $h(d,s)$ is a two-variables function taking values in $R$.
\end{Def}

Then, following the proof of \Cref{A} and \Cref{B}, we have that
\begin{Th}\label{C}
Let $\mathbf{s}=(s_i)_{i=1}^n$, $\mathbf{s}^j:=(s_i)_{i\neq j}$ are defined as above, and $H(\mathbf{s})$ is a Multiple Harmonic Type Sum taking values in a commutative ring $R$. 
If $s_i$ are distinct and $n$ is odd, then we have:
\begin{eqnarray*}
\sum_{\sigma\in S_n}sgn(\sigma)H(\sigma(\mathbf{s}))=\sum_{j=1}^n(-1)^{n-j}H(s_j)\left(\sum_{\sigma\in S_n^j}sgn(\sigma)H(\sigma(\mathbf{s}^j))\right),
\end{eqnarray*}
where $S_n, S_n^j$ are the symmetry group of $\{1,\cdots,n\}$, and $\{1,\cdots,\hat{j},\cdots,n\}$ respectively, in particular, we have $S_n^j\cong S_{n-1}$, sgn is the sign function on $S_n$.
\end{Th}

\begin{Th}\label{D}
Let $\mathbf{s}=(s_i)_{i=1}^n$, $\mathbf{s}^j:=(s_i)_{i\neq j}$ are defined as above, and $H(\mathbf{s})$ is a Multiple Harmonic Type Sum taking values in a commutative ring $R$. If $char(R)=2$, $h(d,2s)=h^2(d,s)$, $s_i$ and $2s_i$ are all distinct(\textit{where $2s_i$
is included only when $k_i>1$}), and $k_i\geq 1$, $\forall i\leq n$. 
Let $\phi:=\sum_{i=1}^nk_i$, and
$\mathbf{s}_0:=((s_1)_{k_1},\cdots,(s_n)_{k_n}).$
If we further have $k_i>1$, define
$\mathbf{s}_i:=((s_1)_{k_1},\cdots,(s_{i})_{k_i-2},2s_i,\cdots,(s_n)_{k_n}).$
Then the following holds:
\begin{eqnarray*}
&&\sum_{{\textit{$i$ where $k_i>1$}}}\sum_{\sigma\in \textbf{L}_i}H(\sigma(\mathbf{s}_i))+\phi\sum_{\sigma\in \textbf{L}_0}H(\sigma(\mathbf{s}_0))\\
&=&\sum_{j, \ where \ k_j\neq 2}H(s_j)\sum_{{\textit{$i\neq j$ where $k_i>1$}}}\sum_{\sigma\in \textbf{L}_i^j}H(\sigma(\mathbf{s}_i^j))\\
&+&\sum_{j, \ where \ k_j> 2}H(s_j)\sum_{\sigma\in \textbf{L}_j^j}H(\sigma(\mathbf{s}_j^j))+\sum_{j,\ where\  k_j>1}H(2s_j)\sum_{\sigma\in \textbf{L}_j^j}H(\sigma(\mathbf{s}_j^{2j}))\\
&+&\sum_{j=1}^nH(s_j)\phi\sum_{\sigma\in \textbf{L}_0^j}H(\sigma(\mathbf{s}_0^j)),
\end{eqnarray*}
where $$\mathbf{s}_i^j:=((s_1)_{k_1},\cdots,(s_{i})_{k_i-2},2s_i,\cdots,(s_j)_{k_j-1},\cdots,(s_n)_{k_n}),$$
$$\mathbf{s}_i^i:=((s_1)_{k_1},\cdots,(s_{i})_{k_i-3},2s_i,\cdots,(s_n)_{k_n}),$$
$$\mathbf{s}_i^{2j}:=((s_1)_{k_1},\cdots,(s_{i})_{k_i-2},2s_i,\cdots,(s_j)_{k_j-1},(s_{j+1})_{k_{j+1}}\cdots,(s_n)_{k_n}),$$
and $\textbf{L}_i^j$ is the set consisting of all possible re-orders of $\mathbf{s}_i^j$, when $k_i>1,j\neq i$ or $k_i>2,j=i$.
\end{Th}

Considering the classical multiple zeta values $\zeta_{\mathbb{Q}}(\mathbf{s})$ and finite multiple zeta values $\zeta_{\mathbb{Q}}^{\mathcal{A}}(\mathbf{s})$ over $\mathbb{Q}$, we can see that they are the sums of Multiple Harmonic Type and hence, by \Cref{C}, we have
\begin{Cor}
Let $\mathbf{s}=(s_i)_{i=1}^n\in \mathbb{N}^n$, $\mathbf{s}^j:=(s_i)_{i\neq j}$ where $\zeta_{\mathbb{Q}}(\mathbf{s})$, $\zeta_{\mathbb{Q}}(\mathbf{s}^j)$ are convergent.
If $s_i$ are distinct and $n$ is odd, then we have
\begin{eqnarray*}
\sum_{\sigma\in S_n}sgn(\sigma)\zeta_{\mathbb{Q}}(\sigma(\mathbf{s}))=\sum_{j=1}^n(-1)^{n-j}\zeta_{\mathbb{Q}}(s_j)\sum_{\sigma\in S_n^j}sgn(\sigma)\zeta_{\mathbb{Q}}(\sigma(\mathbf{s}^j)),
\end{eqnarray*}
where $S_n, S_n^j$ are the symmetry group of $\{1,\cdots,n\}$, and $\{1,\cdots,\hat{j},\cdots,n\}$ respectively, in particular, we have $S_n^j\cong S_{n-1}$, sgn is the sign function on $S_n$.
\end{Cor}

\begin{Cor}
Let $\mathbf{s}=(s_i)_{i=1}^n\in \mathbb{Z}_{\neq 0}^n$, $\mathbf{s}^j:=(s_i)_{i\neq j}$.
If $s_i$ are distinct and $n$ is odd, then we have
\begin{eqnarray*}
\sum_{\sigma\in S_n}sgn(\sigma)\zeta_{\mathbb{Q}}^{\mathcal{A}}(\sigma(\mathbf{s}))=0,
\end{eqnarray*}
where $S_n, S_n^j$ are the symmetry group of $\{1,\cdots,n\}$, and $\{1,\cdots,\hat{j},\cdots,n\}$ respectively, in particular, we have $S_n^j\cong S_{n-1}$, sgn is the sign function on $S_n$.
\end{Cor}
\begin{proof}
This is trivial once we notice that $\zeta_{\mathbb{Q}}^{\mathcal{A}}(\sigma(s_i)=0$ as long as $s_i\neq 0$.
\end{proof}

\section*{acknowledgement}
I am especially thankful to Professor Dinesh S. Thakur for giving lots of valuable suggestions and reviewing the details of this paper.

\end{document}